\pdfoutput=1
\documentclass[english]{jnsao}

\usepackage[utf8]{inputenc}

\usepackage{xspace}
\usepackage{graphicx}
\usepackage{thmtools}
\usepackage{tabularx}
\usepackage{booktabs}
\usepackage{enumitem}
\usepackage{algorithm,algorithmicx,algpseudocode}
\usepackage[nameinlink,noabbrev,capitalize]{cleveref}

\usepackage{siunitx}
\numberwithin{equation}{section}
\newtheorem{assumption}{Assumption}

\newcommand\be{\begin{equation}}
\newcommand\ee{\end{equation}}

\renewcommand{\subset}{\subseteq}

\renewcommand{\d}{\,\text{\rmfamily{}\upshape{}d}}
\newcommand{\dx}{\,\text{\rmfamily{}\upshape{}d}x}

\newcommand{\dt}{\,\text{\rmfamily{}\upshape{}d}t}
\newcommand{\ds}{\,\text{\rmfamily{}\upshape{}d}s}
\def\N{\mathbb  N}
\def\R{\mathbb  R}

\def\L{\mathcal  L}

\def\argmin{\operatorname*{arg\, min}}

\def\prox{\operatorname{prox}}
\def\dom{\operatorname{dom}}

\def\env{\operatorname{env}}

\def\dtol{\delta_{\mathrm{tol}}}
\def\dtolcg{\delta_{\mathrm{CG}}}

\newcolumntype{L}{>{$}l<{$\quad}}
\newcolumntype{R}{>{$}r<{$\quad}}
\newcolumntype{C}{>{$}c<{$}}

\newcommand{\Frechet}{Fr\'echet\xspace}

\manuscriptlicense{CC-BY-SA 4.0}

\manuscripteprinttype{arxiv}
\manuscripteprint{2503.21612}

\manuscriptsubmitted{2025-04-25}
\manuscriptaccepted{2026-03-25}
\manuscriptvolume{6}
\manuscriptnumber{15574}
\manuscriptyear{2026}
\manuscriptdoi{10.46298/jnsao-2026-15574}

\begin{document}

\title{A globalized inexact semismooth Newton method for strongly convex optimal control problems}
\shorttitle{A globalized inexact semismooth Newton method}

\author{Daniel Wachsmuth%
\thanks{Institut f\"ur Mathematik,
Universit\"at W\"urzburg,
97074 W\"urzburg, Germany,
		\email{daniel.wachsmuth@uni-wuerzburg.de},
		\orcid{0000-0001-7828-5614}}}

\shortauthor{Daniel Wachsmuth}

\acknowledgements{
This research was partially funded by the Deutsche Forschungsgemeinschaft (DFG, German Research Foundation) – project number 565876882. 
}

\maketitle

\begin{abstract}
We investigate a globalized inexact semismooth Newton method applied to strongly convex optimization problems in Hilbert spaces.
Here, the semismooth Newton method is appplied to the dual problem, which has a continuously differentiable objective.
We prove global strong convergence of iterates as well as transition to local superlinear convergence.
The latter needs a second-order Taylor expansion involving semismooth derivative concepts.
The convergence of the globalized method is demonstrated in numerical examples, for which
the local unglobalized method diverges.
\end{abstract}

\bigskip


\bigskip

\section{Introduction}

The goal of this article is to develop a globalized semismooth Newton method to solve the following
strongly convex optimization problem:
\begin{equation}\label{eq001}
 \min_{u \in U } \frac12 \|Su-z\|_Y^2  + \frac\alpha2 \|u\|_U^2 + g(u)
\end{equation}
where $U,Y$ are Hilbert spaces, $S \in \L(U,Y)$, $z\in Y$, $\alpha>0$, $g:U \to \bar\R$ is proper, convex, lower semicontinuous with $0 \in \dom g$.

Problem \eqref{eq001} is a nice convex problem and its properties are well understood.
However, the local (unglobalized) semismooth Newton method converges globally only for relatively large values of $\alpha$.

Following \cite{HinzeVierling2012oms,ZhaoSunToh2010,MilzarekSchaippUlbrich2024},
we propose to apply the semismooth Newton method to the dual problem of \eqref{eq001},
where the globalization uses the dual objective as merit function.
Due to the structure of \eqref{eq001}, the dual problem has a continuously differentiable objective functional,
which enables a particularly nice convergence analysis.
As this analysis crucially relies on the strong convexity of the primal problem, it cannot be extended to non-convex or not strongly convex problems.
Hence, the proposed method is intended to be used as a subproblem solver in, e.g., proximal point methods \cite{MilzarekSchaippUlbrich2024}
or iterated Bregman methods \cite{PornerWachsmuth2016}.

Interestingly, in the implementation only knowledge about the Moreau-Yosida envelope and the proximal operator of $g$ is needed.
Using the proximal operator, one can prove that $u \in U$ solves \eqref{eq001}
if and only if  $u \in U$ solves the non-smooth equation
\begin{equation}\label{eq002}
 u = \prox_{g/\alpha} \left( -\frac1\alpha S^*(Su-z) \right),
\end{equation}
where $S^*$ denotes the Hilbert space adjoint of $S$.
In order to solve \eqref{eq001} numerically, the most common approach is to apply the semismooth Newton method to the non-smooth equation \eqref{eq002}.
In this paper, we propose to solve the equation
\begin{equation}\label{eq003}
\xi - z + S  \prox_{g/\alpha} ( S^* ( \xi/\alpha) ) = 0,
\end{equation}
which is related to \eqref{eq002} via $\xi = z - Su$. The advantage of this approach is that the derivative of the operator appearing in \eqref{eq003}
is a self-adjoint map, which is not the case for \eqref{eq002}.
In addition, \eqref{eq003} is equivalent to $\nabla \Phi(\xi)=0$, where $\Phi$ is the dual function to \eqref{eq001}, which will be used as merit function.

A third approach to solve \eqref{eq001} is based on the so-called normal map \cite{Robinson1992}, where one solves
\begin{equation}\label{eq004}
p - S^*(S \prox_{g/\alpha}( - p/\alpha) - z) = 0
\end{equation}
for $p\in U$. Clearly, $p$ is related to $\xi$ by $p = - S^*\xi$.
Again the derivative of the operator appearing in \eqref{eq004} is not self-adjoint.

This artice is written 25 years after the publication of the seminal contributions \cite{Ulbrich2002}  (submitted 2000) and \cite{HintermullerItoKunisch2002} (submitted 2001)
on semismooth Newton methods applied to optimal control problems.
Still the literature on globalized semismooth Newton methods for infinite-dimensional problems is rather scarce.
In fact, to the best of our knowledge, only the following contributions develop a global convergent method applicable to \eqref{eq001}.
In \cite{HinzeVierling2012oms} the globalization based on \eqref{eq003} was investigated for the special case, where $g$ is the indicator function associated with box constraints.
In the monograph \cite{Ulbrich2011} a trust-region globalization of the semi-smooth Newton method applied to the non-smooth equation \eqref{eq002}
is analyzed. Since continuous differentiability of $g$ is required, it is not applicable to choices like $g(u) = \|u\|_{L^1(\Omega)}$.
A similar method was analyzed for the special case of a parabolic control problem in \cite{KelleySachs1999}.
In \cite{GerdtsHornKimmerle2017} a globalized method is introduced, where the convergence analysis relies on strong assumptions on the sequence of chosen derivatives.
For $\ell^1$-regularized problems, a globalized semismooth Newton method can be found in \cite{HansRaasch2015}. 
One major drawback of the convergence analysis in \cite{Ulbrich2011,GerdtsHornKimmerle2017,HansRaasch2015} is that they assume the existence of strong accumulation points of the sequence of iterates.
While it is natural to expect that the sequence of iterates is bounded, this only guarantees the existence of {\em weakly} converging subsequences.
However, {\em strong} convergence of iterates has to be verified in order to prove transition to fast local convergence.

Let us comment on related work on globalized semismooth Newton methods.
The recent work \cite{alphonse2024globalizedinexactsemismoothnewton} introduces a semismooth Newton method to solve contractive fixed point problems.
Due to the contraction assumption, no damping or globalization is needed.
In \cite{PotzlSchielaJaap2022,PotzlSchielaJaap2024} globalized proximal Newton methods that solve minimization problems in Hilbert spaces are investigated.
The use of the normal map equation was proposed in \cite{HinzeVierling_2012spp} together with a damped Newton method without convergence analysis.
Global convergence of the unregularized semismooth Newton method applied to problems with box constraints was proven in
\cite{BergouniouxItoKunisch1999, HintermullerItoKunisch2002} for discretized problems and \cite{KunischRosch2002} in the infinite-dimensional setting provided $\alpha$ is large relative to $\|S\|$.
For a review of early results on (globalized) semismooth Newton methods we refer to \cite{QiSun1999}.
Recent contributions
on globalized semismooth Newton methods to minimize the sum of a differentiable and a non-smooth function in finite dimensions can be found in
\cite{OuyangMilzarek2024,Gfrerer2025}.

The goal of this article is to give an almost self-contained analysis of the global convergence of the globalized semismooth Newton method.
A few facts about subdifferentials and proximal operators are collected in \cref{sec_proximals}.
The dual problem to \eqref{eq001} is introduced in \cref{sec_dual}.
Due to the problem structure, no duals have to be computed for the implementation, only the proximal operator of $g$ is needed.
For the transition to fast convergence, a second order Taylor expansion is essential, which is introduced in \cref{sec_2ndorder}.
As in  \cite{ZhaoSunToh2010,MilzarekSchaippUlbrich2024} we use a conjugate gradient method as solver for the Newton equations.
In order to prove that all iterates of the CG method are descent directions, one elementary result is needed, which is provided in \cref{sec_cg}.
The main section is \cref{sec_ssn}, which contains the globalized semismooth Newton method \cref{alg_ssn}.
We prove {\em strong} convergence of the iterates in \cref{lem_weak_convergence}, and the transition to fast local convergence in \cref{thm_superlinear_convergence}.
\cref{sec_nemyzki} shows that commonly studied examples admit a second order Taylor expansion.
As we show in \cref{sec_vd}, the application of the semismooth Newton method to \eqref{eq003} enables the use of the variational discretization \cite{Hinze2005}
similar to the arguments in \cite{HinzeVierling_2012spp}.
The paper concludes with numerical experiments in \cref{sec_numerics}. The results indicate that the variational discretization leads to better performance of
the semismooth Newton method for small values of $\alpha$ in some examples.


\section{Preliminaries on the primal and dual problems}

\begin{lemma}
 The problem \eqref{eq001} is uniquely solvable.
 An element $\bar u \in U$ is a solution of \eqref{eq001} if and only if
 \begin{equation} \label{eq_primal_noc}
  0 \in S^*(S\bar u - z ) + \alpha \bar u + \partial g(\bar u).
  \end{equation}
\end{lemma}
\begin{proof}
Existence can be proven by classical arguments,
see, e.g., \cite[Theorem 11.10]{BauschkeCombettes2017}.
Let $\bar u$ be a solution of \eqref{eq001}.
Since the first two addends of \eqref{eq001} are continuous, we can apply the subgradient sum-rule twice to obtain \eqref{eq_primal_noc}.
The sufficiency of \eqref{eq_primal_noc} follows from convexity.
\end{proof}

\subsection{Conjugate functions and proximal operators}
\label{sec_proximals}

\begin{definition}
 Let $X$ be a Hilbert space, $h : X \to \bar\R$ proper, convex, lower semicontinuous.
 The convex conjugate $h^*$ of $h$ is defined by
 \[
  h^*(\zeta) := \sup_{x\in X} \langle \zeta,x\rangle - h(x).
 \]
 The Moreau-Yosida envelope of $h$ is defined by
 \[
  \env_h(\zeta) := \min_{x\in X} \frac12\|x-\zeta\|_X^2 + h(x),
 \]
 and the proximal operator associated to $h$ is defined by
 \[
  \prox_h (\zeta) := \argmin_{x\in X} \frac12\|x-\zeta\|_X^2 + h(x).
 \]
\end{definition}

For properties of Moreau-Yosida envelopes and  proximal operators, we refer to \cite[Section 12.4]{BauschkeCombettes2017}.
Convex conjugates are developed in \cite[Chapter 13]{BauschkeCombettes2017}.

\begin{proposition}[{\cite[Proposition 12.26]{BauschkeCombettes2017}}]
\label{prop_prox_characterization}
 Let $X$ be a Hilbert space, $h : X \to \bar\R$ proper, convex, lower semicontinuous.
 Let $x,\zeta \in X$. Then $x = \prox_h (\zeta)$ if and only if $0 \in x-\zeta + \partial h(x)$.
\end{proposition}

\begin{proposition}[{\cite[Proposition 12.28]{BauschkeCombettes2017}}]
\label{prop_prox_lipschitz}
 Let $X$ be a Hilbert space, $h : X \to \bar\R$ proper, convex, lower semicontinuous.
 Let $\zeta_1,\zeta_2 \in X$. Then $\|\prox_h(\zeta_1)-\prox_h(\zeta_2)\|_X \le \|\zeta_1-\zeta_2\|_X$.
\end{proposition}

\begin{proposition}[{\cite[Proposition 12.30]{BauschkeCombettes2017}}]
\label{prop_diff_env}
 Let $X$ be a Hilbert space, $h : X \to \bar\R$ proper, convex, lower semicontinuous.
 Then $\env_h : X \to \R$ is \Frechet differentiable with
 \[
  \nabla \env_h(\zeta) = \zeta - \prox_h(\zeta).
 \]
\end{proposition}

\begin{lemma}\label{lem_diff_conjugate}
Let $X$ be a Hilbert space, $h : X \to \bar\R$ proper, convex, lower semicontinuous.
Define
\[
 H^*:= \left( \frac12\|\cdot\|_X^2 + h \right)^*.
\]
Then $H^*:X \to \R$ satisfies
\[
 H^*(\zeta) =    \frac12\|\zeta\|_X^2-\env_h(\zeta)
\]
and is \Frechet differentiable with
\[
 \nabla H^*(\zeta) = \prox_h(\zeta) .
\]
If $0 \in \dom h$ then $H^*$ is bounded from below.
\end{lemma}
\begin{proof}
By definition, we have
\[\begin{split}
 H^*(\zeta) &= \sup_{x\in X} \langle \zeta,x\rangle - \frac12\|x\|_X^2 -h(x) \\
 &= - \inf_{x\in X} \left( \frac12\|x-\zeta\|_X^2 +h(x) - \frac12\|\zeta\|_X^2 \right)\\
 & = -\env_h(\zeta) + \frac12\|\zeta\|_X^2.
\end{split}\]
Due to \cref{prop_diff_env}, it follows that $H^*$ is \Frechet differentiable, and the expression of $\nabla H^*$ follows.
Boundedness from below is a direct consequence of the definition of $H^*$.
\end{proof}

\subsection{Optimality conditions for the dual problem}
\label{sec_dual}

We define the functions $F:Y \to \R$  and $G:U \to \bar \R$ by
\[
 F(y) =  \frac12\|y-z\|_Y^2, \quad G(u) := \frac\alpha2\|u\|_U^2 + g(u).
\]
Let us introduce the Fenchel dual problem to \eqref{eq001}, see, e.g., \cite[Section 15.3]{BauschkeCombettes2017}.
The dual problem of \eqref{eq001} is given by
\begin{equation}\label{dual}
 \min_{\xi \in Y} F^*(-\xi) + G^*( S^*\xi) =: \Phi(\xi)
\end{equation}
where
\[
F^*(\zeta)= \frac12\|\zeta+z\|_Y^2 - \frac12\|z\|_Y^2
\]
and
\[
 G^*(u) = \frac1{2\alpha} \|u\|_U^2 - \alpha \env_{g/\alpha}(u/\alpha).
\]
Due to the structure of the problem both functions $F^*$ and $G^*$ can be evaluated without the need to compute conjugates, only the Moreau-Yosida envelope of $g/\alpha$ is needed.
%

\begin{lemma}\label{lem_dual} 
The dual problem \eqref{dual} is
uniquely
solvable.
Let $\bar\xi \in Y$. Then $\bar\xi$ solves \eqref{dual} if and only if
\[
 \nabla \Phi(\bar\xi) = \bar\xi - z + S\prox_{g/\alpha}(S^*\bar\xi/\alpha) = 0.
\]
\end{lemma}
\begin{proof}
The function $\Phi$ is continuous and convex. Since $F^*$ is coercive, existence of solutions can be proven easily.
Uniqueness of minimizers is a consequence of the strict convexity of $F^*$.
The dual problem \eqref{dual} is the minimization problem of the convex and \Frechet differentiable function $\Phi$.
Then $\nabla \Phi(\bar\xi) = 0$ is necessary and sufficient for optimality of $\bar\xi$.
The expression of the derivative is a consequence of \cref{prop_diff_env}.
%
\end{proof}

\begin{lemma}\label{lem_convert_primal_dual}
If $\bar u\in U$ solves the primal problem \eqref{eq001} then $\bar\xi:=-(S\bar u -z)$ solves  \eqref{dual}.
Conversely, if $\bar\xi \in Y$ solves \eqref{dual} then $\bar u:=\prox_{g/\alpha}(S^*\bar\xi/\alpha)$ solves \eqref{eq001}.
\end{lemma}
\begin{proof}
Due to \cref{prop_prox_characterization}, the optimality condition \eqref{eq_primal_noc} is equivalent to
 \[
\bar u = \prox_{g/\alpha}\left( -\frac1\alpha S^*(S\bar u-z) \right).
 \]
Then the claim follows by elementary computations.
\end{proof}


\section{Second-order semismoothness}
\label{sec_2ndorder}

For the proof of transition from global convergence to fast local convergence, a second-order Taylor expansion of the function $G^*$ is indispensible \cite{Facchinei1995}.
Following \cite{PotzlSchielaJaap2022}, we introduce the notion of second-order semismoothness.

Let $X$ be a Banach space.
We will denote the space of bilinear forms from $X \times X$ to $\R$  by $\L^{(2)}(X,\R)$.

\begin{definition}\label{def_second_order_semismooth}
 Let $T: X \to \R$ be continuously differentiable with derivative $T'$.
 Let $\partial^2 T$ be a set-valued map $\partial^2 T: X \rightrightarrows \L^{(2)}(X,\R)$ with non-empty images.
 Then $T$ is called second-order semismooth
 with respect to $\partial^2 T$ if for all $x\in X$
 \begin{equation} \label{eq_def_semismooth_2}
  \sup_{M \in \partial^2 T(x+h)} \left|T(x+h) - T(x) - T'(x)h - \frac12  M(h,h)\right| = o(\|h\|_X^2)
 \end{equation}
 and
 \begin{equation} \label{eq_def_semismooth_1}
  \sup_{M \in \partial^2 T(x+h)} \left\|T'(x+h) - T'(x)  -  M(h,\cdot)\right\|_{X^*} = o(\|h\|_X)
 \end{equation}
 for $\|h\|_X\to 0$,
 where $M(h,\cdot)$ denotes the element of $X^*$ given by $v \mapsto M(h,v)$.
\end{definition}

Here, second-order semismoothness refers to the availability of a second-order Taylor expansion \eqref{eq_def_semismooth_2},
which is different to the use in \cite{Ulbrich2011}, where higher-order semismoothness refers to stronger estimates of the first-order expansion \eqref{eq_def_semismooth_1}
replacing $o(\|h\|_X)$ for $O(\|h\|_X^{1+\alpha})$ with $\alpha \in (0,1]$.

If $X$ is a Hilbert space, we can use operators $M \in \L(X,X)$ instead of bilinear forms, and replace $M(h,h)$ and $M(h,\cdot)$ by $\langle Mh,h\rangle$ and $Mh$.
In \cite{PotzlSchielaJaap2022} an example is presented to show that \eqref{eq_def_semismooth_2} does not imply \eqref{eq_def_semismooth_1}.
Let us prove a sufficient condition for \eqref{eq_def_semismooth_2} which is only based on assumptions on $T'$.

\begin{lemma}\label{lem_2nd_order_semi}
 Let $T: X \to \R$ be continuously differentiable with
 Bouligand differentiable $T'$, i.e., $T'$ is  directionally differentiable and
 for all $x\in X$ it holds
 \[
 \|T'(x+h) - T'(x) - T''(x;h)\|_{X^*} = o(\|h\|_X).
 \]
Let
 $\partial^2 T: X \rightrightarrows \L^{(2)}(X,\R)$ be given such that \eqref{eq_def_semismooth_1} is satisfied for all $x\in X$.
 Then \eqref{eq_def_semismooth_2} is satisfied for all $x\in X$, and hence, $T$ is second-order semismooth
 with respect to $\partial^2 T$.
\end{lemma}
\begin{proof}
For $M \in \partial^2 T(x+h)$, we have the expansion
\begin{multline*}
 T(x+h) - T(x) - T'(x)h - \frac12  M(h,h) = \\
 \left( T(x+h) - T(x) - \frac12 [T'(x)+T'(x+h)] h \right) \\
 + \frac12 \Big( T'(x+h)h - T'(x)h  -  M(h,h)\Big).
\end{multline*}
Due to \eqref{eq_def_semismooth_1}, the second addend is of order $o(\|h\|_X^2)$ for $h\to0$.
Hence, it suffices to prove
\[
   \left|T(x+h) - T(x) - \frac12 [T'(x)+T'(x+h)] h \right| = o(\|h\|_X^2).
\]
Take $x\in X$.
Let $\epsilon>0$.
Then there is $\rho>0$ such that  $\|T'(x+h) - T'(x) - T''(x;h)\|_{X^*} \le \epsilon \|h\|_X$  for all $h$ with $\|h\|_X< \rho$.
Take $h$ with $\|h\|_X <\rho$.
Denote
\[
 R_t := T'(x+th) - T'(x) - tT''(x;h),
\]
which implies $\|R_t\|_{X^*} \le \epsilon t\|h\|_X$ for all $t\in [0,1]$.
Then we can write
\[
\begin{aligned}
 T(x+h) - T(x) - \frac12 [T'(x)+T'(x+h)] h
 &= \int_0^1 [ T'(x+sh) -  T'(x)]h \ds - \frac12[ T'(x+h)-T'(x)]h\\
 &= \int_0^1 [ sT''(x;h) +R_s]h \ds  - \frac12T''(x;h)-\frac12R_1h\\
 &= \int_0^1 R_sh \ds - \frac12 R_1h
  \end{aligned}
 \]
so that
\[
   \left|T(x+h) - T(x) - \frac12 [T'(x)+T'(x+h)] h \right|  \le \epsilon \|h\|_X^2
\]
which finishes the proof.
\end{proof}

In the smooth case, we have the following result.

\begin{corollary}[{\cite[Proposition 5]{PotzlSchielaJaap2022}}]
 Let $T:  X \to \R$ be twice continuously differentiable. Then $T$ is second-order semismooth with respect to $\partial^2 T(x) := \{ T''(x)\}$.
\end{corollary}


\section{Conjugate Gradients}
\label{sec_cg}

Let us recall the celebrated conjugate gradient method to solve linear systems with positive definite operators.
Let $X$ be a Hilbert space, $A \in \L(X,X)$ self-adjoint and positive definite, and $b\in X$.
The  CG algorithm to solve $Ax=b$ in the notation from \cite[Section 38]{TrefethenBau1997} can be found in \cref{alg_CG}.

\begin{algorithm}
 \begin{algorithmic}
  \State{Input: tolerance $\dtolcg>0$.}
  \State{Set $x_0=0$, $r_0=b$, $p_0=r_0$, $k=0$.}
  \While{$\|r_k\| > \dtolcg$} 
    \State{$\alpha_k = \frac{\|r_k\|_X^2}{\langle Ap_k,p_k\rangle}$}
    \State{$x_{k+1} = x_k + \alpha_k p_k$}
    \State{$r_{k+1} = r_k - \alpha_k Ap_k$}
    \State{$\beta_k = \frac{\|r_{k+1}\|_X^2}{\|r_k\|_X^2}$}
    \State{$p_{k+1} = r_{k+1} + \beta_k p_k$}
    \State{Set $k:=k+1$}
  \EndWhile
 \end{algorithmic}
 \caption{CG algorithm to solve $Ax=b$ with tolerance $\dtolcg$} \label{alg_CG}
\end{algorithm}

The classical orthogonality relations and convergence guarantees transfer to the Hilbert space case, see, e.g., \cite{Daniel1967}.
If $A=I+K$ with $K$ compact we have r-superlinear convergence of $x_k \to A^{-1}b$ \cite{Winther1980}.
These results imply that for positive tolerance $\dtolcg$, algorithm \cref{alg_CG} terminates after finitely many steps.

In the following, we need an estimate of $ \langle x_k,b\rangle$.
A similar result can be found in \cite[Lemma 3.1]{ZhaoSunToh2010}. We give a simplified proof.

\begin{lemma}\label{lem_cg_estimate}
Let $A \in \L(X,X)$ be self-adjoint and positive definite, and $b\in X$.
Let $(x_k)$ be iterates of \cref{alg_CG}. Then
\[
  \langle x_k,b\rangle  \ge \frac1{\|A\|_{\L(X,X)}} \|b\|_X^2.
\]
\end{lemma}
\begin{proof}
For $k=1$, we have $x_0=0$ and $p_0=b$ so that
\[
\langle x_1 , b\rangle = \alpha_0 \|b\|_X^2 = \frac{\|b\|_X^4}{\langle Ab,b\rangle} \ge \frac1{\|A\|_{\L(X,X)}} \|b\|_X^2.
\]
Due to the orthogonality relations of the residuals, \cite[(38.4)]{TrefethenBau1997}, we have $\langle r_k,b \rangle = \langle r_k,r_0 \rangle =0$ for all $k\ge1$.
Then we get for the conjugate directions
\[
 \langle p_{k+1},b \rangle  = \langle r_{k+1},b \rangle + \beta_k \langle p_k,b \rangle  = \beta_k \langle p_k,b \rangle,
\]
so that $\langle p_{k+1},b \rangle = \prod_{i=0}^k \beta_i \langle p_0,b \rangle$.
Since $p_0 =r_0=b$ and $\beta_i>0$, this implies $\langle p_k,b \rangle\ge0$ for all $k\ge0$.
Using this in the update formula for $x_{k+1}$, we find
\[
 \langle x_{k+1},b \rangle = \langle x_k,b \rangle + \alpha_k\langle p_k,b \rangle \ge  \langle x_k,b \rangle,
\]
so that the sequence $( \langle x_k,b \rangle)$ is monotonically increasing.
\end{proof}


\section{Semismooth Newton}
\label{sec_ssn}

Here, we want to investigate a semismooth Newton method to minimize the function $\Phi$ in \eqref{dual}, which can be written as
\begin{equation} \label{def_Phi}
\Phi(\xi) = \frac12 \|\xi-z\|_Y^2 -\frac12\|z\|_Y^2 +  \frac1{2\alpha} \|S^*\xi\|_U^2 - \alpha \env_{g/\alpha}(S^*\xi/\alpha).
\end{equation}
Its gradient is given by \cref{lem_dual} as
\begin{equation} \label{def_DPhi}
 \nabla \Phi(\xi) = \xi - z + S\prox_{g/\alpha}(S^*\xi/\alpha) .
\end{equation}

\begin{corollary}\label{cor_DPhi_Lipschitz}
The map $\Phi$ is continuously differentiable. In addition, there is $L>0$ such that $\| \nabla \Phi(y_1) - \nabla \Phi(y_2)\|_Y \le L \|y_1-y_2\|_Y$ for all $y_1,y_2 \in Y$.
Consequently,
\[
 \left| \Phi(y+h) - \Phi(y) - \Phi'(y)h\right| \le \frac L2 \|h\|_Y^2.
\]
\end{corollary}
\begin{proof}
 Differentiability of $\Phi$ is a consequence of the differentiability of $\env_{g/\alpha}$ by \cref{prop_diff_env}.
 By \cref{prop_prox_lipschitz}, the proximal operator $\prox_{g/\alpha}$ is Lipschitz continuous.
 Hence, $\nabla \Phi$ is Lipschitz continuous.
 The claimed estimate follows is then a consequence of this Lipschitz property.
\end{proof}

Let us define for convenience
\begin{equation} \label{def_Psi}
 \Psi( \xi ):= \frac1{2\alpha} \|S^*\xi\|_U^2 - \alpha \env_{g/\alpha}(S^*\xi/\alpha),
\end{equation}
so that $\nabla\Psi(\xi) = S\prox_{g/\alpha}(S^*\xi/\alpha)$.
Throughout this section, we assume that the following assumption is satisfied.

\begin{assumption}\label{ass_semismooth}
There is a set-valued map $\partial^2 \Psi : Y \rightrightarrows \mathcal B$,
where
\[
 \mathcal B := \{M \in \L(Y,Y): \ M=M^*, \ \langle My,y\rangle \ge 0 \ \forall y \in Y\},
\]
such that $\Psi$ given by \eqref{def_Psi} is second-order semismooth with respect to $\partial^2 \Psi$, see \cref{def_second_order_semismooth}.
In addition, $\partial^2 \Psi$ is bounded, i.e., there is $c>0$ such that for all $y \in Y$ and $M\in \partial^2 \Psi (y)$ it holds
 \[
 \|M\|_{\L(Y,Y)} \le c.
 \]
\end{assumption}

Let us emphasize that in the algorithm we will not need to know the precise values of the Lipschitz constant $L$ or  the bound $c$.
For the particular case of Nemyzki operators on $L^2$ spaces, validity of \cref{ass_semismooth} is discussed in \cref{sec_nemyzki}.

\begin{corollary}\label{cor_D2Phi}
 $\Phi$ is second-order semismooth with respect to $\partial^2 \Phi := I + \partial^2 \Psi$.
\end{corollary}

\begin{corollary}\label{cor_bound_M}
 Let $M \in  \partial^2\Phi(y)$ for some $y \in Y$. Then $M \in \L(Y,Y)$ is continuously invertible with $\|M\|_{\L(Y,Y)} \le 1+c$ and $\|M^{-1}\|_{\L(Y,Y)} \le 1$.
 Moreover, $\langle Md,d\rangle \ge \|d\|_Y^2$ for all $d\in Y$.
\end{corollary}
The core of the semismooth Newton methods is the Newton equation to solve a linearization of  $ \nabla \Phi(\xi)=0$.
Given an iterate $\xi_k \in Y$, we choose $M_k\in\partial^2 \Psi(\xi_k)$,
and determine $d_k \in Y$ as a solution of the Newton equation
\begin{equation}\label{eq_ssn_xi}
 (I + M_k)d_k = -  \nabla \Phi(\xi_k) = -\left(\xi_k - z + S\prox_{g/\alpha}(S^*\xi_k/\alpha) \right).
\end{equation}
In the upcoming convergence analysis, we will use both remainder estimates \eqref{eq_def_semismooth_2} and \eqref{eq_def_semismooth_1}. The first-order
estimate is an ingredient of the classical local convergence proof of semismooth Newton \cite{Ulbrich2011}, while the second-order estimate is
needed in the proof of transition to fast local convergence with full steps ($t_k=1$).

\begin{algorithm}[htbp]
\begin{algorithmic}
 \State Choose $\dtol>0$,
 $\sigma\in(0,1/2)$, $\beta\in (0,1)$,
$\eta\ge0$,
 $\tau\in(0,1]$.
\Comment{Parameters}

\State Choose $\xi_0 \in Y$, set $k:=0$.
\Comment{Initialization}

 \While{ $\|\nabla \Phi(\xi_k)\|_Y > \dtol$}
 \State Choose $M_k \in \partial^2\Phi(\xi_k) $, see \cref{cor_D2Phi}.

 \State Compute $d_k \in Y$  such that   \Comment{Inexact solve of Newton equation}
     \begin{equation}\label{eq_alg_inexact}
      \| M_k d_k + \nabla \Phi(\xi_k)\|_Y \le  \eta\|\nabla \Phi(\xi_k)\|_Y^{1+\tau}
     \end{equation}
   \State   using \cref{alg_CG} with tolerance \Comment{CG}
   \[ \dtolcg := \eta\|\nabla \Phi(\xi_k)\|_Y^{1+\tau}. \]

    \State Compute the smallest $l_k\in \N_0$ such that\Comment{Armijo linesearch}
     \begin{equation}\label{eq_alg_armijo}
     \Phi(\xi_k + \beta^{l_k} d_k) - \Phi(\xi_k) \le \sigma \beta ^{l_k}  \langle d_k, \nabla \Phi(\xi_k) \rangle
     \end{equation}

    \State Set $t_k:=\beta^{l_k}$, $\xi_{k+1} := \xi_k + t_k d_k$, $k:=k+1$.

 \EndWhile

\end{algorithmic}
\caption{Globalized inexact semismooth Newton method}
\label{alg_ssn}
\end{algorithm}

\cref{alg_ssn} is basically \cite[Algorithm 4.1]{MilzarekSchaippUlbrich2024}. The convergence proof in the finite-dimensional case is given in \cite[Theorem 4.1]{MilzarekSchaippUlbrich2024}.
For the special case of box constraints and exact Newton solves ($\eta=0$), \cref{alg_ssn} is equivalent to \cite[Algorithm 4.2]{HinzeVierling2012oms}.

%
%
%

\begin{lemma}
 \cref{alg_ssn} is well-defined.
\end{lemma}
\begin{proof}
Suppose $\nabla  \Phi(\xi_k)\ne0$.
 Due to the special form of the Newton derivatives $M_k \in \partial^2\Phi(\xi_k) $ in \cref{cor_D2Phi}, it follows directly that
 the operators $M_k$ are self-adjoint, positive definite, and invertible.
 The direction $ -M_k^{-1} \nabla \Phi(\xi_k)$, satisfies \eqref{eq_alg_inexact}.
 Then \cref{alg_CG} terminates after finitely many steps.
 Due to \cref{lem_cg_estimate} the returned element $d_k$ is a
 descent direction to the \Frechet differentiable function $\Phi$.
 Under this condition, the Armijo linesearch terminates after finitely many steps.
\end{proof}

The termination criterion is motivated by the following bound, see also \cite[Lemma 4.3 (iii)]{HinzeVierling2012oms}.

\begin{lemma}\label{lem_error_bound}
 Let $\bar\xi\in Y$ be such that $\nabla \Phi(\bar\xi)=0$. Let $y \in Y$. Then
 \[
  \|y-\bar\xi\|_Y \le \|\nabla\Phi(y)\|_Y.
 \]
\end{lemma}
\begin{proof}
Since $y \mapsto   S\prox_{g/\alpha}(S^*\bar\xi/\alpha)$ is monotone as the derivative of the convex function $\Psi$, it follows
\[
 \langle \nabla \Phi(y) - \nabla \Phi(\bar\xi), y-\bar\xi\rangle \ge \|\bar\xi-y\|_Y^2,
\]
which proves the claim using the Cauchy-Schwarz inequality.
\end{proof}

The inexact Newton directions are descent directions of $\Phi$. In addition,
we have the following estimate on the directional derivative.

\begin{corollary}\label{cor_66}
Let $(\xi_k)$ be a sequence of iterates of \cref{alg_ssn} together with the search directions $(d_k)$ satisfying \eqref{eq_alg_inexact}.
Then it holds
     \begin{equation}\label{eq_alg_descent}
     \langle d_k, \nabla \Phi(\xi_k) \rangle \le -\frac1{1+c}\|\nabla \Phi(\xi_k)\|_Y^2,
    \end{equation}
    where $c$ is from \cref{ass_semismooth}.
\end{corollary}
\begin{proof}
\cref{lem_cg_estimate} implies $\langle d_k,  - \nabla \Phi(\xi_k) \rangle \ge \frac1{ \|M_k\|_{\L(Y,Y)} } \|\nabla \Phi(\xi_k)\|_Y^2$,
and the claim follows with \cref{cor_bound_M}.
\end{proof}

In the sequel, we assume that the algorithm performs infinitely many steps and generates a sequence of iterates $(\xi_k)$.
We start the convergence analysis with the following basic lemma, which uses the properties of the Armijo linesearch.

\begin{lemma} \label{lem_basic}
 Let $(\xi_k)$ be a sequence of iterates of \cref{alg_ssn} together with the step sizes $(t_k)$ and search directions $(d_k)$.
 Then
 \[
  t_k\langle d_k, \nabla \Phi(\xi_k) \rangle \to 0.
 \]
\end{lemma}
\begin{proof}
 Due to \cref{lem_diff_conjugate}, $G^*$, $\Psi$, and $\Phi$ are bounded from below,
 hence $( \Phi(\xi_k))$ is monotonically decreasing and converging.
 The claim is now a consequence of  \eqref{eq_alg_armijo}.
%
%
%
\end{proof}

Since we have not proven until now that $(\xi_k)$ has  strongly converging subsequences, we cannot conclude
like in finite-dimensions, where $\nabla \Phi(\xi_k) \to0$ along a strongly converging subsequence follows from uniform continuity of $\nabla \Phi$.

Using \cref{cor_DPhi_Lipschitz}, we will prove a lower bound of $t_k$ next.
There, we will use the global Lipschitz continuity of $\nabla \Phi$. If we would know that $(\xi_k)$ has  strongly converging subsequences
then local Lipschitz continuity of $\nabla \Phi$ would be enough.

\begin{lemma} \label{lem_tk}
 Let $(\xi_k)$ be a sequence of iterates of \cref{alg_ssn} together with the step sizes $(t_k)$ and search directions $(d_k)$.
Then there is a constant $\sigma'>0$ such that $t_k<1$ implies
\[
 t_k \ge \sigma' \frac{  |\langle d_k, \nabla \Phi(\xi_k) \rangle| }{  \|d_k\|_Y^2 }.
\]
\end{lemma}
\begin{proof}
Let $t>0$. Then
\[
 \Phi(\xi_k + t d_k) - \Phi(\xi_k) - \sigma t  \langle d_k, \nabla \Phi(\xi_k) \rangle
 \le (1- \sigma) t  \langle d_k, \nabla \Phi(\xi_k) \rangle + \frac{Lt^2}2 \|d_k\|_Y^2.
\]
The expression on the right-hand side is non-positive for all $t$ satisfying
\[
 0 \le t \le  (1- \sigma) \frac{  |\langle d_k, \nabla \Phi(\xi_k) \rangle| }{ \frac L2 \|d_k\|_Y^2 }.
\]
If $t_k<1$ satisfies the Armijo condition \eqref{eq_alg_armijo}
then $l_k\ge1$, and $\beta^{l_k-1}=t_k/\beta$ violates the Armijo condition \eqref{eq_alg_armijo}.
Consequently, $t_k/\beta$ is larger than the quantity on the right-hand side in the above inequality,
which implies the claim.
\end{proof}

Now we are in the position to prove $\nabla \Phi(\xi_k) \to0$.

\begin{lemma}\label{lem_convergence_DPhi}
 Let $(\xi_k)$ be a sequence of iterates of \cref{alg_ssn}. Then it holds  $\liminf_{k\to\infty} t_k >0$  and $\nabla \Phi(\xi_k) \to 0$ in $Y$.
\end{lemma}
\begin{proof}
By \cref{lem_basic}, we have $t_k\langle d_k, \nabla \Phi(\xi_k) \rangle \to 0 $.
 If $\liminf_{k\to\infty} t_k >0$ then \eqref{eq_alg_descent} implies $\|\nabla \Phi(\xi_k)\|_Y \to 0$.
 Consider the case $\liminf_{k\to\infty} t_k=0$. By extracting a subsequence, we can assume $t_k \to 0$ and $t_k<1$ for all $k$.
 Then from $t_k \to0$ and $t_k\langle d_k, \nabla \Phi(\xi_k) \rangle \to 0 $ we get with the help of \cref{lem_tk}
 \[
   \frac{  |\langle d_k, \nabla \Phi(\xi_k) \rangle| }{  \|d_k\|_Y^2 } +
   \frac{  |\langle d_k, \nabla \Phi(\xi_k) \rangle|^2 }{  \|d_k\|_Y^2 } \to 0,
 \]
so that  \eqref{eq_alg_descent} implies
 \[
\frac{  \|\nabla \Phi(\xi_k)\|_Y^2 }{  \|d_k\|_Y^2 }  +
\frac{  \|\nabla \Phi(\xi_k)\|_Y^4 }{  \|d_k\|_Y^2 }  \to 0.
 \]
From the inexactness condition \eqref{eq_alg_inexact} and  $\|M_k^{-1}\|_{\L(Y,Y)} \le 1$, we get
\[
\|d_k\|_Y \le \|\nabla \Phi(\xi_k)\|_Y + \eta \|\nabla \Phi(\xi_k)\|_Y^{1+\tau}.
\]
Dividing by $\|d_k\|_Y$ and passing to the limit $k\to\infty$ leads to the contradiction $1\le 0$.
Here, we used $\tau\le1$, and $\frac{a^{1+\tau}}b = \left( \frac{a^2}b\right)^\tau \left( \frac ab\right)^{1-\tau} $.
\end{proof}


Due to the structure of the problem, we have strong convergence of iterates.

\begin{lemma} \label{lem_weak_convergence}
 Let $(\xi_k)$ be a sequence of iterates of \cref{alg_ssn}.
 Then $\xi_k \to \bar\xi$ in $Y$, where $\xi$ is the unique solution of $\Phi(\xi)=0$.
 In addition, $d_k \to 0$ in $Y$.
\end{lemma}
\begin{proof}
Let $ \bar\xi$ be such that $\Phi( \bar\xi)=0$. The existence of such $ \bar\xi$ follows from \cref{lem_dual}.
\cref{lem_error_bound} implies $\| \bar\xi - \xi_k\|_Y \le \|\nabla \Phi(\xi_k)\|_Y$, the latter quantity tends to zero by \cref{lem_convergence_DPhi}.
\end{proof}

The next step is to prove that the step size $t_k=1$ is accepted for all $k \ge k_0$. Here, we follow the exposition in \cite{Facchinei1995}.
First, we establish superlinear convergence of the local inexact semismooth Newton method.

\begin{lemma}\label{lem_small_o}
Let $ \bar\xi$ be the unique solution of $\Phi( \bar\xi)=0$.
Let $(\xi_k)$ be a sequence of iterates of \cref{alg_ssn} with step sizes $(t_k)$.
Let $r_k := M_k d_k + \nabla \Phi(\xi_k)$. Then
\[
 \|r_k\|_Y = o(\|\xi_k -  \bar\xi\|_Y), \qquad \|\xi_k + d_k -  \bar\xi\|_Y= o(\|\xi_k -  \bar\xi\|_Y)
\]
for $k \to \infty$.
\end{lemma}
\begin{proof}
Using the inexact Newton equation \eqref{eq_alg_inexact} and \cref{cor_DPhi_Lipschitz}, we get
\[
 \|r_k\|_Y \le \eta\|\nabla \Phi(\xi_k)\|_Y^{1+\tau} \le \eta L \| \xi_k-\bar\xi\|_Y^{1+\tau} = o(\|\xi_k - \bar\xi\|_Y).
\]
Using again \eqref{eq_alg_inexact}, we find
\[\begin{split}
 \|\xi_k + d_k - \bar\xi\|_Y &= \|M_k^{-1}( M_k(\xi_k - \bar\xi) - \nabla \Phi(\xi_k) + r_k)\|_Y\\
 &\le \|M_k^{-1}\|_{\L(Y,Y)} \|M_k(\xi_k - \bar\xi) - \nabla \Phi(\xi_k)\|_Y + \|r_k\|_Y\\
 & = o(\|\xi_k - \bar\xi\|_Y),
 \end{split}
\]
where we have used \cref{cor_bound_M} and the semismoothness of $\nabla \Phi$.
\end{proof}

Now, we will prove that the step size $t_k=1$ will be accepted for large $k$. This proof crucially relies on the second-order
expansion in the definition of second-order semismoothness (\cref{def_second_order_semismooth}).

\begin{lemma} \label{lem_step}
Let $\bar\xi$ be the unique solution of $\Phi(\bar\xi)=0$.
 Let $(\xi_k)$ be a sequence of iterates of \cref{alg_ssn} with step sizes $(t_k)$.
 Then there is $k_0>0$ such that $t_k =1$ for all $k \ge k_0$.
\end{lemma}
\begin{proof}
See also \cite[Theorem 3.3]{Facchinei1995}.
Using the simple inequalities
\[
  \|\xi_k - \bar\xi\|_Y \le \|d_k\|_Y +  \|\xi_k + d_k - \bar\xi\|_Y \le \|\xi_k - \bar\xi\|_Y + 2\|\xi_k + d_k - \bar\xi\|_Y,
\]
we get by \cref{lem_small_o}
\[
 \|d_k\|_Y =   \|\xi_k - \bar\xi\|_Y +o(\|\xi_k - \bar\xi\|_Y).
\]
Similarly, we get from \eqref{eq_alg_inexact} an estimate of the directional derivative
\[\begin{split}
 \langle d_k, \nabla \Phi(\xi_k) \rangle & \le -\langle M_kd_k,d_k\rangle +\|r_k\|_Y \|d_k\|_Y \\
 &\le - \|d_k\|_Y^2 +o(\|\xi_k - \bar\xi\|_Y^2)\\
 &= - \|\xi_k - \bar\xi\|_Y^2 +o(\|\xi_k - \bar\xi\|_Y^2).
\end{split}\]
Since $\xi_k+d_k \to \bar\xi$ in $Y$, we have the second-order expansion
\[ \begin{aligned}
\Phi(\xi_k+d_k) - \Phi(\bar\xi) -\langle\nabla \Phi(\bar\xi), \xi_k + d_k - \bar\xi\rangle
 & = - \frac12 \langle H_k( \xi_k + d_k - \bar\xi), \xi_k + d_k - \bar\xi\rangle + o(\|\xi_k + d_k - \bar\xi\|_Y^2) \\
& = o(\|\xi_k - \bar\xi\|_Y^2),
\end{aligned}\]
where $H_k \in \partial^2\Phi(\bar\xi_k+d_k)$, which is bounded due to \cref{cor_D2Phi}.
Similarly, we prove using $\nabla \Phi(\bar\xi)=0$
\[
 \Phi(\xi_k) - \Phi(\bar\xi) = - \frac12 \langle M_k( \xi_k   - \bar\xi), \xi_k   - \bar\xi\rangle + o(\|\xi_k - \bar\xi\|_Y^2).
\]
Collecting all these estimates yields
\begin{multline*}
 \Phi(\xi_k+d_k) - \Phi(\xi_k) -\frac12 \langle \nabla \Phi(\xi_k),d_k\rangle \\
 \begin{aligned}
  & =\frac12 \langle M_k( \xi_k   - \bar\xi), \xi_k  - \bar\xi\rangle -\frac12\langle \nabla \Phi(\xi_k),d_k\rangle + o(\|\xi_k - \bar\xi\|_Y^2) \\
  & =\frac12 \langle M_k( \xi_k   - \bar\xi), \xi_k +d_k  - \bar\xi-d_k\rangle -\frac12\langle \nabla \Phi(\xi_k),d_k\rangle + o(\|\xi_k - \bar\xi\|_Y^2) \\
  & =-\frac12 \langle M_k( \xi_k   - \bar\xi)+\nabla \Phi(\xi_k) - \nabla \Phi(\bar\xi),  d_k\rangle  + o(\|\xi_k - \bar\xi\|_Y^2) \\
  & = o(\|\xi_k - \bar\xi\|_Y^2).
 \end{aligned}
\end{multline*}
Using this in the descent condition \eqref{eq_alg_armijo}, we get
\[\begin{aligned}
      \Phi(\xi_k +   d_k) - \Phi(\xi_k) - \sigma   \langle d_k, \nabla \Phi(\xi_k) \rangle
&
      =\left( \frac12 - \sigma\right) \langle d_k, \nabla \Phi(\xi_k) \rangle + o(\|\xi_k - \bar\xi\|_Y^2) \\
 &     = - \left( \frac12 - \sigma\right) \|\xi_k - \bar\xi\|_Y^2 + o(\|\xi_k - \bar\xi\|_Y^2) .
\end{aligned}\]
The right-hand side is less than zero for $k$ large enough, which proves the claim.
\end{proof}

\begin{theorem} \label{thm_superlinear_convergence}
 Let \cref{ass_semismooth} be satisfied.
Let $\bar\xi$ be the unique solution of $\nabla\Phi(\bar\xi)=0$.
 Let $(\xi_k)$ be a sequence of iterates of \cref{alg_ssn}.
 Then $\xi_k \to \bar\xi$ q-superlinearly.
\end{theorem}
\begin{proof}
 Due to \cref{lem_step}, there is   $k_0>0$ such that $t_k =1$ for all $k \ge k_0$.
 Then \cref{lem_small_o} implies $ \|\xi_k + d_k - \bar\xi\|_Y = \|\xi_{k+1}-\bar\xi\|_Y= o(\|\xi_k - \bar\xi\|_Y)\|$ for $k\ge k_0$,
 which is the claim.
\end{proof}

Let us conclude this section with several remarks.


\begin{remark}
 The proof of transition to fast convergence in \cite[Lemma 4.4]{HinzeVierling2012oms} relied on stronger assumptions.
 There a condition on the measure of almost active sets was used, which implies strict complementarity.
 In addition, the operator $S^*$ has to be continuous with values in $C^1(\bar\Omega)$.
 Under these conditions $\partial^2 \Psi$ is single-valued and H\"older continuous in a neighborhood of the solution \cite[Lemma 3.4]{HinzeVierling2012oms},
 which allows to prove the higher-order convergence rate $\|\xi_{k+1} - \bar\xi\|_Y \le c \|\xi_k-\bar\xi\|_Y^{3/2}$.
\end{remark}

\begin{remark}[Newton equation in primal form]
\label{rem_ssn_primal}
The element $\bar\xi \in Y$ solves
\[
  0 = \nabla \Phi(\bar\xi) = \bar\xi - z + S\prox_{g/\alpha}(S^*\bar\xi/\alpha)
\]
if and only if $\bar u = \prox_{g/\alpha}(S^*\bar\xi/\alpha)$  solves
\begin{equation}\label{eq_noc_u}
 0 = \bar u - \prox_{g/\alpha}(-S^*(S\bar u-z)/\alpha) .
\end{equation}
Conversely, if $\bar u$ is a solution of \eqref{eq_noc_u} then $\bar\xi = z- S\bar u$ satisfies $\nabla \Phi(\bar\xi)=0$.
Let $u_k\in U$ be given. Applying the semismooth Newton method to \eqref{eq_noc_u} amounts to choosing $H_k \in \partial \prox_{g/\alpha}(-S^*(Su_k-z)/\alpha) $
and solving
\begin{equation}\label{eq_ssn_u}
 \delta u + \frac 1\alpha H_k S^*S \delta u = - \left( u_k - \prox_{g/\alpha}(-S^*(Su_k-z)/\alpha) \right).
\end{equation}
%
Let us set $\xi_k := z- Su_k$. If $\delta u$ solves \eqref{eq_ssn_u} then $\delta \xi:= -S \delta u$ solves
\[
 \delta u - \frac 1\alpha H_k S^* \delta \xi = - \left( u_k - \prox_{g/\alpha}(S^*\xi_k/\alpha) \right).
\]
Multiplying by $-S$, we get
\[
 \delta \xi + \frac 1\alpha S H_k S^* \delta \xi = - \left( \xi_k -z + S \prox_{g/\alpha}(S^*\xi_k/\alpha) \right),
\]
which is the Newton equation \eqref{eq_ssn_xi} with $M_k := S^*H_kS$.
Equation \eqref{eq_ssn_u} is an equation with an unsymmetric operator. If $H_k$ is suitably structured, one
can transform the equation into an equation with a symmetric operator.
\end{remark}

\begin{remark}[Semismooth Newton applied to normal map]
\label{rem_ssn_normaleq}
It is easy to show that $\bar u = \prox_{g/\alpha}( - \bar p/\alpha)$ is a solution of \eqref{eq001} if and only if $\bar p\in U$ solves
\[
 0 = \bar p - S^*(S \prox_{g/\alpha}( - \bar p/\alpha) - z).
\]
The map $p \mapsto p - S^*(S \prox_{g/\alpha}( - p/\alpha) - z)$ is called normal map after \cite{Robinson1992}.
Semismooth Newton methods based on this equation are considered in \cite{PieperDiss,HinzeVierling_2012spp}.
In the recent article \cite{OuyangMilzarek2024}, the convergence analysis of a trust-region globalization of a semismooth Newton method applied to the normal map equation is
developed in finite dimensions, where globalization is based on the residual in the normal equation.
\end{remark}

\begin{remark}In \cite{Graser2008} a globalization of the active-set strategy \cite{BergouniouxItoKunisch1999} to solve discretized box constrained optimal control problems
was proposed. However, the analysis does not carry over to the infinite-dimensional setting, as certain operators are not invertible in infinite dimensions \cite[Remark 3]{Graser2008}.
\end{remark}

\begin{remark}[Non-quadratic functional]
Let us briefly consider the following generalization of \eqref{eq001}:
 \[
  \min_{u\in X} F(Su) + \frac\alpha2 \|u\|_U^2 + g(u),
 \]
where $F : Y \to \R$ is convex. Let $F^*$ denote the convex conjugate of $F$.
Under the assumption that $F$ is strongly convex and twice continuously differentiable, we get that $F^*$ is also strongly convex and twice continuously differentiable.
Let us denote with $\nabla F^* := \nabla ( F^*)$ its gradient.
Then, we have $\nabla F( \nabla F^*(y)) = y$ for all $y\in Y$ by \cite[Theorem 16.29]{BauschkeCombettes2017}.
Differentiating this identity yields
$
 \nabla^2 F(  \nabla F^*(y)) \nabla^2 F^*(y)) = \operatorname{id}
$ 
or $\nabla^2 F^*(y) = \nabla^2 F(  \nabla F^*(y)) ^{-1}$.
%
%
%
The resulting second-order derivative of $\Phi$ is given by $\partial^2 \Phi = \nabla^2 F^* + \partial^2\Psi$.
Then all the convergence analysis for the corresponding semismooth Newton algorithm remains true, only the constants in \cref{lem_error_bound} and \cref{cor_bound_M}
have to be adapted.
\end{remark}


\section{Semismoothness of Nemyzki operators}
\label{sec_nemyzki}

Let $(\Omega,\mathcal A,\mu)$ be a measure space.
In many optimization problems, the function $g$ in \eqref{eq001} is defined as
\begin{equation}\label{eq_def_g}
g(u)= \int_\Omega \tilde g(u(x))\d\mu
\end{equation}
with a convex function $\tilde g:\R\to\bar\R$.
Here, we follow the convention to set $g(u):=+\infty$ if $x \mapsto \tilde g(u(x))$ is not integrable.
The resulting proximal operator is then a superposition operator on $U = L^2(\Omega)$.

\begin{corollary}
 Let $\tilde g: \R\to\bar\R$  be proper, convex, and lower semicontinuous. Let $g$ be defined in \eqref{eq_def_g}.
 Suppose there is $u \in L^2(\Omega)$ with $g(u) < +\infty$.
 Then $\prox_g : L^2(\Omega) \to L^2(\Omega)$ is given by
 \[
  \prox_g( v) (x) = \prox_{\tilde g}( v(x) ) \quad \text{ for $\mu$-a.a. } x\in \Omega.
 \]
\end{corollary}
\begin{proof}
 This follows from the results of \cite[Section 3]{Rockafellar1976},
 in particular the characterization of $\partial g(u)$ as
 \[
  \partial g(u) = \{ d \in L^2(\Omega): \ d(x) \in \partial \tilde g(u(x)) \text{ for $\mu$-almost all } x\in \Omega\},
 \]
see \cite[Corollary 3E]{Rockafellar1976}.
\end{proof}

\subsection{Piecewise smooth proximals}

We will now work with piecewise smooth proximal operators.

\begin{assumption}\label{ass_prox_piecewise}
The function  $\prox_{\tilde g/\alpha}: \R \to \R$ is piecewise smooth in the following sense: there are real numbers $v_1 < v_2 < \dots < v_N$ and continuously
differentiable functions $p_k:\R\to\R$, $k=0\dots N$,
such that
\[
 \prox_{\tilde g/\alpha} = p_k \text{ on } [v_k,v_{k+1}],
\]
where we set $v_0=-\infty$, $v_{N+1}:=+\infty$.
\end{assumption}

\begin{lemma}\label{lem_prox_2nd_order}
Let $\prox_{\tilde g/\alpha}$ satisfy \cref{ass_prox_piecewise}. Define
the single-valued map
\[
 \partial \prox_{\tilde g/\alpha} (v) :=
  p_k'(v)   \text{ if }  v\in [v_k,v_{k+1}) \text{ for some } k =0\dots N.
\]
Then for all $v\in \R$, we have
\[
 \left|\prox_{\tilde g/\alpha}(v+h) - \prox_{\tilde g/\alpha}(v)  -  \partial\prox_{\tilde g/\alpha}(v+h)h \right| = o(|h|)
\]
and
\[
|\prox_{\tilde g/\alpha}(v+h) - \prox_{\tilde g/\alpha}(v) - \prox_{\tilde g/\alpha}'(v;h)|= o(|h|)
.
\]
\end{lemma}
\begin{proof}
 Take $v\in \R$, and $h$ such that $\prox_{\tilde g/\alpha}$ is continuously differentiable between $v$ and $v+h$.
 Then both claims follow easily.
\end{proof}

Clearly, this result can be generalized for $\prox_{\tilde g/\alpha}: \R^n \to \R^n$ with infinitely many pieces,
see also \cite[Section 2.5.3]{Ulbrich2011}.


In the next step, we prove that the derivative of the proximal operator of the integral functional $g$
is semismooth from $L^q(\Omega)$ to $L^p(\Omega)$ for $p<q$.
The proof of the following result is adapted from the proof of \cite[Proposition 6]{PotzlSchielaJaap2022},
more general results can be found in \cite[Section 3.3]{Ulbrich2011}.

\begin{lemma} \label{lem_prox_2nd_order_LpLq}
Let $1\le p<q\le\infty$ be given. Let $\mu(\Omega)<\infty$.
Let $\prox_{\tilde g/\alpha}$ satisfy \cref{ass_prox_piecewise}.
Let $v \in L^q(\Omega)$.
Define the measurable function $\partial \prox_{g/\alpha}(v)$ for $\mu$-almost all $x\in \Omega$ by
\[
 \partial \prox_{g/\alpha} (v)(x) :=  \partial \prox_{\tilde g/\alpha}(v(x)).
\]
Then we have $\partial \prox_{g/\alpha} (v) \in L^\infty(\Omega)$,
\[
  \left\|\prox_{g/\alpha}(v+h) - \prox_{g/\alpha}(v)  -  \partial \prox_{g/\alpha} (v+h) \cdot h \right\|_{L^p(\Omega)} = o(\|h\|_{L^q(\Omega)})
\]
and
\[
  \|\prox_{g/\alpha}(v+h) - \prox_{g/\alpha}(v) - \prox_{g/\alpha}'(v;h)\|_{L^p(\Omega)} = o(\|h\|_{L^q(\Omega)}).
\]
\end{lemma}
\begin{proof}
Take $v\in L^q(\Omega)$.
Let us define for $x\in \Omega$ and $h \in \mathbb \R\setminus\{0\}$
\[
 r(x,h) := \frac1h \left( \prox_{\tilde g/\alpha}(v(x)+h) - \prox_{\tilde g/\alpha}(v(x))  -  \partial\prox_{\tilde g/\alpha}(v(x)+h)h \right).
\]
In addition, we set $r(x,0)=0$ for $x\in\Omega$.
Then $x\mapsto r(x,h(x))$ is $\mu$-measurable for all measurable $h:\Omega \to \R$. In addition, $\lim_{h\to 0} r(x,h)=0$ for almost all $x$ by \cref{lem_prox_2nd_order}.
Since $\prox_{\tilde g/\alpha}$ is Lipschitz (\cref{prop_prox_lipschitz}) and consequently $|\partial \prox_{\tilde g/\alpha}|\le 1$,
we have for almost all $x$ and all $h$ the estimate $|r(x,h)| \le 2$.

Let $(h_n)$ be a sequence in $L^q(\Omega)$ such that $h_n \to 0$ in $L^q(\Omega)$. Then there
is a pointwise $\mu$-a.e.\ converging subsequence $(h_{n_k})$. 
It follows $r(x,h_{n_k}(x)) \to 0$ for almost all $x$. Since $r$ is uniformly bounded and $\mu(\Omega)<\infty$, we get $\|r(\cdot,h_{n_k})\|_{L^s(\Omega)} \to 0$
for all $s\in [1,\infty)$ by dominated convergence.
It follows
\begin{multline*}
 \left\|\prox_{g/\alpha}(v+h_{n_k}) - \prox_{g/\alpha}(v)  -  \partial \prox_{g/\alpha} (v+h_{n_k})\cdot h_{n_k} \right\|_{L^p(\Omega)} \\
 =  \|r(\cdot,h_{n_k})h_{n_k} \|_{L^p(\Omega)}
 \le \|r(\cdot,h_{n_k})\|_{L^s(\Omega)} \|h_{n_k} \|_{L^q(\Omega)}
\end{multline*}
from the generalized H\"older inequality
with $s<\infty$ such that $\frac1p = \frac1s + \frac1q$.
Consequently,
\[
\frac{ \left\|\prox_{g/\alpha}(v+h_{n_k}) - \prox_{g/\alpha}(v)  -  \partial \prox_{g/\alpha} (v+h_{n_k})\cdot h_{n_k} \right\|_{L^p(\Omega)} }{ \|h_{n_k} \|_{L^q(\Omega)} } \to 0.
\]
Since the limit is independent of the chosen subsequence, the convergence along the whole sequence $(h_k)$ follows.

The second claim can be proven with similar arguments, observing that $|\prox_{\tilde g/\alpha}'(v;h)| \le |h|$.
\end{proof}

Note that we cannot get rid of the norm gap in the previous result.
The necessity of this gap has been discussed in \cite[Example 3.57]{Ulbrich2011}, \cite[Example 8.14]{ItoKunisch2008} for semismoothness
and in \cite[Remark 4.4]{GriesseGrundWachsmuth2008}, \cite{Malanowski2003}
for Bouligand differentiability.

\begin{example}
Without the assumption $\mu(\Omega)<\infty$, the result of \cref{lem_prox_2nd_order_LpLq} is not true.
Let
\[
\tilde g(v)  = \begin{cases} + \infty & \text{ if } v<0, \\
             0 & \text{ if } v\ge0. \\
            \end{cases}
\]
Then $\prox_{\tilde g}(v) = \max(v,0)$.
Let $\Omega=(0,+\infty)$, and let $\mu$ be the Lebesgue measure. Take $v(x):= -\exp(-x)$, so that $v\in L^p(\R)$ for all $p\in [1,+\infty]$.
Define $h_n := \chi_{[n,+\infty)} (-2v)$. Then $\|h_n\|_{L^p(\Omega)} = \exp(-n) \|v\|_{L^p(\Omega)}$, so that it follows
$\lim_{n\to\infty}\|h_n\|_{L^p(\Omega)}\to0$ for all $p\in [1,+\infty]$.
In addition,
\[
\prox_g(v+h_n) - \prox_g(v+h_n) - \prox_g'(v;h+h_n) =  \chi_{[n,+\infty)}(v+h_n) -0-0= \frac12 h_n.
\]
It is easy to check that $ \|h_n\|_{L^p(\Omega)} / \|h_n\|_{L^q(\Omega)}$ does not converge to zero for any choice of $p,q$.
\end{example}

Due to the norm gap appearing in \cref{lem_prox_2nd_order_LpLq}, we cannot expect that $\prox_{g/\alpha}$ is semismooth as a mapping on $L^2(\Omega)$.
In order that \cref{ass_semismooth} is satisfied, we will need to rely on smoothing properties of $S$ or $S^*$.

\begin{lemma}\label{lem_DDPsi}
Assume $\mu(\Omega)< \infty$. Set $U:= L^2(\Omega)$.
Let $\prox_{\tilde g/\alpha}$ satisfy \cref{ass_prox_piecewise}.
Let $Y$ be a Hilbert space, $S \in \L( L^2(\Omega),Y)$ be given such that the Hilbert space adjoint $S^*$ belongs to $\L(Y,L^q(\Omega))$ for some $q>2$.

Then the map $\Psi: Y \to Y$ defined by \eqref{def_Psi}
is second-order semismooth with respect to $\partial^2 \Psi:Y \rightrightarrows \L(Y,Y)$ defined by
\begin{equation}\label{def_DDPsi}
 \partial^2 \Psi(\xi):=\left\{ \frac1\alpha S \partial\prox_{g/\alpha}(S^*\xi/\alpha) S^* \right\},
\end{equation}
where $\partial\prox_{g/\alpha}$ is defined in \cref{lem_prox_2nd_order_LpLq}.
Moreover, \cref{ass_semismooth} is satisfied.
\end{lemma}

In \eqref{def_DDPsi}, the function $\partial\prox_{g/\alpha}$ is meant to act as multiplication operator on $S^*$, i.e.,
\[
\left(\frac1\alpha S \partial\prox_{g/\alpha}(S^*\xi/\alpha) S^*\right)(v) = \frac1\alpha S \left( \partial\prox_{g/\alpha} \cdot S^*v \right),
\]
where $\cdot$ denotes the pointwise multiplication of $\partial\prox_{g/\alpha} \in L^\infty(\Omega)$ and $S^*v \in L^q(\Omega)$.

\begin{proof}[Proof of \cref{lem_DDPsi}]
We show that $ \partial^2 \Psi(\xi)$ from \eqref{def_DDPsi} satisfies the assumptions of \cref{lem_2nd_order_semi}.
Let us recall $\nabla \Psi(\xi) = S\prox_{g/\alpha}(S^*\xi/\alpha)$. Let $v,h\in Y$ be given. Then \cref{lem_prox_2nd_order_LpLq} with $p=2$
and the assumption on $S^*$ imply
\begin{multline*}
 \left\| \prox_{g/\alpha}(S^*(v+h)) - \prox_{g/\alpha}(S^*v) - \partial\prox_{g/\alpha}(S^*(v+h)) \cdot S^*h \right\|_{L^2(\Omega)} \\
 = o( \|S^*h\|_{L^q(\Omega)} ) = o( \|h\|_Y).
\end{multline*}
Clearly this implies that $\partial^2 \Psi(\xi)$ satisfies \eqref{eq_def_semismooth_1}.
The Bouligand property of the directional derivative can be proven with analogous arguments.
\end{proof}

\subsection{Examples}
\label{sec_examples}

\paragraph{Box constraints.}

Let $R>0$ be given.
Let us define $\tilde g$ by
\[
 \tilde g_{\textrm{box}} (u) = \begin{cases}
                0 & \text{ if } |u|\le R, \\
                +\infty & \text{ if } |u|>R.
               \end{cases}
\]
Then $\int_\Omega \tilde g_{\textrm{box}}(u) \dx< +\infty$ if and only if $|u(x)| \le R$ $\mu$-almost everywhere, that is if $u$ satisfies the box constraints.
In this case, the proximal operator is equal to the projection onto the associated feasible set,
\[
 \prox_{\tilde g_{\textrm{box}}}(v) = \max(-R, \min(+R, v)).
\]
Clearly, this function is piecewise affine linear, and \cref{ass_prox_piecewise} is satisfied.
For this particular choice of $\tilde g$, we have $\tilde g_{\textrm{box}} / \alpha = \tilde g_{\textrm{box}}$.

\paragraph{$L^1$-norm.}

Let us now consider the case
\[
 \tilde g_{L^1}(u) = \beta |u|.
\]
Then $\int_\Omega \tilde g_{L^1}(u) \d\mu =\beta \|u\|_{L^1(\Omega)}$. The corresponding proximal operator is the so-called soft-thresholding operator given by
\[
 \prox_{\tilde g_{L^1}}(v ) = \max(0, v-\beta) + \min(0, v+\beta) =
    \begin{cases}
    v+\beta & \text{ if } v < -\beta,\\
    0       & \text{ if } |v| \le \beta,\\
    v-\beta & \text{ if } v > \beta.
    \end{cases}
\]
This is again a piecewise affine linear map, so that \cref{ass_prox_piecewise} is satisfied.
The proximal operator of $ \tilde g_{L^1}/\alpha$ is obtained by replacing $\beta$  with $\beta/\alpha$ in the formula for $\prox_{\tilde g_{L^1}}$.

\paragraph{Box constraints plus $L^1$-norm.}

If $\tilde g = \tilde g_{L^1}(u) + \tilde g_{\textrm{box}} (u)$, then it is not hard to check that
\[ 
 \prox_{\tilde g}(v) =  \max(0, \min(v-\beta,R)) + \min(0, \max(v+\beta,-R))
 =
    \begin{cases}
    -R & \text{ if } v < -\beta-R\\
    v+\beta & \text{ if }-\beta-R\le v < -\beta,\\
    0       & \text{ if } |v| \le \beta,\\
    v-\beta & \text{ if }\beta\le v < \beta+R,\\
    R & \text{ if } v > \beta+R,
    \end{cases}
\]
and hence,  \cref{ass_prox_piecewise} is satisfied.
Again,
the proximal operator of $( \tilde g_{L^1}+ \tilde g_{\textrm{box}})/\alpha$ is obtained by replacing $\beta$  with $\beta/\alpha$ in the formula above.

\paragraph{$L^2$-ball constraint.}

Let us now consider $g$ to be the indicator function of a  ball in $L^2(\Omega)$ with radius $\gamma>0$:
\[
 g(u) = \begin{cases}
         0 & \text{ if } \|u\|_{L^2(\Omega)} \le \gamma\\
         +\infty & \text{ if } \|u\|_{L^2(\Omega)} > \gamma.
        \end{cases}
\]
Again, the corresponding proximal map is the projection on the ball with radius $\gamma$, $\prox_g(v) = v \max(1, \frac{\gamma}{\|v\|_{L^2(\Omega)} })$.
Let us choose $\partial \prox_g(v) \in \L( L^2(\Omega),L^2(\Omega))$ by
\[
 \partial \prox_g(v)h :=\begin{cases}
h & \text{ if } \|v\|_{L^2(\Omega)} \leq \gamma\\
\frac{\gamma h}{\|v\|_{L^2(\Omega)}} - \frac{\gamma v \langle v,h\rangle_{L^2(\Omega)} }{ \|v\|_{L^2(\Omega)}^3} & \text{ if } \|q\|_{L^2(\omega)} > \gamma.
 \end{cases}
\]
Then $T' = \prox_g$ and $T'' = \partial \prox_g$ satisfy condition \eqref{eq_def_semismooth_1} by \cite[Proposition 4.1]{KunischWachsmuth2013}.
In addition, $\prox_g$ is Bouligand differentiable due to \cite[Proposition 3.4]{KunischWachsmuth2013}.

\paragraph{Directional sparsity.}

Here, $U = L^2( \Omega; L^2(I))$ where $I = (0,T)$.
\[
 g(u) = \beta \int_\Omega \left( \int_I |u(x,t)|^2 \dt\right)^{1/2} \dx
\]
The proximal operator associated to $g$ is given by
\[
 \prox_{g/\alpha}(u)(x,t) = \max\left(0, \|u(x,\cdot)\|_{L^2(I)} - \frac\beta\alpha\right) \frac{ u(x,t) } { \|u(x,\cdot)\|_{L^2(I)} }.
\]
Here, we used the notation
\[
\|u(x,\cdot)\|_{L^2(I)} :=  \left( \int_I |u(x,t)|^2 \dt\right)^{1/2}.
\]
This operator is semismooth from $L^p( \Omega; L^2(I))$ to $L^2( \Omega; L^2(I))$ for all $p>2$,
\cite[Lemma 3.22]{PieperDiss}, \cite[Lemma 3.2]{HerzogStadlerWachsmuth2012}.
In addition, it is easy to check that $\prox_{g/\alpha}$ is Bouligand differentiable using the chain rule for the Bouligand derivative \cite[Corollary A.4]{Robinson1987}.

\paragraph{Further examples.}

Additional examples that satisfy these assumptions are multi-bang problems \cite{ClasonKunisch2014},
where $g$ is given as the convex hull of
\[
 u \mapsto \begin{cases}
            \beta u^2 & \text{ if } u \in \mathbb Z,\\
            +\infty & \text{ otherwise.}
           \end{cases}
\]
In addition, also vector-valued controls can be considered \cite{ClasonTamelingWirth2021}.
Since $u = \prox_{g/\alpha}(v)$  if and only if $0 \in \alpha(u-v) + \partial g(u)$,
implicit function theorems \cite{Kruse2018} can be used to prove the semismoothness of proximal maps.


\section{Variational discretization}
\label{sec_vd}

For computations, one has to choose finite-dimensional approximations of the infinite-dimensional problem \eqref{eq001}.
The application of the semismooth Newton method to discretized problems is straight-forward.
In this section, we will focus on the variational discretization of \cite{Hinze2005}.
We will consider the following optimal control problem: Minimize
\[
\frac12 \|y - z\|_{L^2(\Omega)}^2 + \frac\alpha2 \|u\|_{L^2(\Omega)}^2 + \int_\Omega \tilde g(u(x))\dx
\]
over all $(y,u) \in H^1_0(\Omega) \times L^2(\Omega)$
subject to
\[
 \int_\Omega \nabla y \cdot \nabla v \dx = \int_\Omega u v \dx \quad \forall v \in H^1_0(\Omega),
\]
where $\Omega\subset \R^d$ is a bounded domain, $z\in L^2(\Omega)$ is given, and $\tilde g: \R \to \bar \R$ is as in \cref{sec_nemyzki}.
The state space is discretized using piecewise linear and continuous finite element functions:
let $V_h \subset H^1_0(\Omega)$ be the corresponding finite element subspace.
Given $u \in L^2(\Omega)$, let $S_hu :=y_h \in V_h$
denote the solution of
\begin{equation} \label{eq_state_discretized}
 \int_\Omega \nabla y_h \cdot \nabla v_h \dx = \int_\Omega u v_h \dx \quad \forall v_h \in V_h,
\end{equation}
which gives rise to a linear and bounded operator $S_h : L^2(\Omega) \to V_h$. Note that
the control space $U = L^2(\Omega)$ is not discretized.

The dual of the discretized problem is now posed in $V_h$: find $\xi_h \in V_h$
such that
\begin{equation} \label{eq_dual_discretized}
 \xi - z + S_h\prox_{g/\alpha}(S_h^*\xi/\alpha)  =0.
\end{equation}
The crucial observation of \cite{Hinze2005} is the following: in this equation $S_h$ is applied to elements of the type $\prox_{g/\alpha}( q_h )$,
where $q_h$ is a piecewise linear finite element function. That is, if $\prox_{g/\alpha}$ has a favorable structure, then $S_h \prox_{g/\alpha}( q_h )$ can be evaluated
numerically even if $\prox_{g/\alpha}( q_h )$ does not belong to a finite element subspace.
This is the case for the first four examples discussed in \cref{sec_examples}.
For instance, if $g$ models box constraints on the control, then $\prox_{g/\alpha}( q_h )$ is piecewise linear.
This piecewise linear function has kinks in the interior of
elements in general, nevertheless the right-hand side for $u=\prox_{g/\alpha}( q_h )$ in \eqref{eq_state_discretized} can be computed exactly \cite{Hinze2005,HinzePinnauUlbrichUlbrich2009}.

In addition, when applying the semismooth Newton method to the dual equation \eqref{eq_dual_discretized}, these kinks do not accumulate.
This is not the case when applying the semismooth Newton to the primal equation as in \cref{rem_ssn_primal}.
To overcome this difficulty, \cite{HinzeVierling_2012spp} suggested to apply the semismooth Newton method to the normal map equation as in \cref{rem_ssn_normaleq}.
The variational discretization is not limited to piecewise linear finite elements, in \cite{SevillaWachsmuth2010} the application to finite elements of polynomial degree 2 is discussed.


\section{Numerical experiments}
\label{sec_numerics}
We will consider the following optimal control problem: Minimize
\begin{equation} \label{eq_exp_primal}
J(y,u):=\frac12 \|y - z\|_{L^2(\Omega)}^2 + \frac\alpha2 \|u\|_{L^2(\Omega)}^2 + \int_\Omega \tilde g(u(x))\dx
\end{equation}
over all $(y,u) \in H^1_0(\Omega) \times L^2(\Omega)$
subject to
\begin{equation} \label{eq_weak_formulation}
 \int_\Omega \nabla y \cdot \nabla v \dx = \int_\Omega u v \dx \quad \forall v \in H^1_0(\Omega),
\end{equation}
where $\Omega =(0,1)^2$, $z\in L^2(\Omega)$ is given, and $\tilde g: \R \to \bar \R$ is as in \cref{sec_nemyzki}.
Here, \eqref{eq_weak_formulation} is the weak formulation of $-\Delta y = u$ with homogeneous Dirichlet boundary conditions.

We set $U := L^2(\Omega)$ and $Y:= L^2(\Omega)$.
We define $Su := y$, where $y\in H^1_0(\Omega)$ is the uniquely determined solution of \eqref{eq_weak_formulation}.
Then $S\in \L(Y,U)$ is bounded, moreover, $S^*$ is linear and continuous from $L^2(\Omega) = Y$ to $L^\infty(\Omega)$.

For discretization we used standard piecewise linear and conforming finite elements on a triangulation of $\Omega$ to discretize states ($y$) and dual quantities ($\xi$).
Controls ($u$) were discretized with piecewise constant functions. To evaluate $\prox_{g/\alpha}(S^*\xi/\alpha)$, first the piecewise linear function $S^*\xi$ is projected
onto the subspace of piecewise constant functions, so that the proximal map is computed for this projection.
All implementations were done in Matlab.

\subsection{Example 1: box constraints and \texorpdfstring{$\scriptstyle L^1$}{L¹}-cost}

Here, we work with
\[
 \tilde g(u) = \beta |u| + \begin{cases}
                     +\infty & \text{ if } |u| > R, \\
                     0 & \text{ if } |u| \le R \\
                    \end{cases}
\]
with
\[
\beta = 10^{-2}, \ R = 1000.
\]
The parameter $\alpha$ is set to  $\alpha = 10^{-5}$ if not mentioned otherwise.
The (local) semismooth Newton method for this problem was analyzed in \cite{Stadler2009}.
As argued in \cref{sec_examples}, $g$ satisfies \cref{ass_prox_piecewise}. Due to the smoothing properties of $S^*$,
the assumptions of \cref{lem_DDPsi} are satisfied, so that \cref{ass_semismooth} holds for this example.
As desired state we used
\[
 z(x_1,x_2) = 10x_1\sin(5x_1)\cos(7x_2).
 \]
We always start the iterations with $\xi_0 = -z$, which corresponds to the initial guess $u_0=0$ for the control.
For this problem, the unglobalized semismooth Newton method does not converge, hence globalization is necessary.

In the implementation, we used the following values for the parameters of \cref{alg_ssn}:
\[
\sigma=0.1,\ \beta =0.5, \   \tau=1,\ \dtol= 10^{-12}.
\]
In our implementation we used the following inexactness criterion
\[
 \| M_k d_k + \nabla \Phi(\xi_k)\|_Y
 \le \min\left( 10^{-4} ,\ 0.1 \|\nabla \Phi(\xi_k)\|_Y,\ \|\nabla \Phi(\xi_k)\|_Y^2 \right),
\]
which is slightly stronger than the one which is used in the convergence analysis,
and is inspired by the choices taken in \cite{MilzarekSchaippUlbrich2024}.
In addition, we stopped the iteration if the threshold in the linesearch is smaller than the machine epsilon, i.e.,
\begin{equation}\label{eq_dual_accuracy}
 \left|  \langle d_k, \nabla \Phi(\xi_k) \rangle \right| \le \texttt{eps(dual)},
\end{equation}
where the Matlab expression {\ttfamily eps(dual)}
returns the positive distance from {\ttfamily dual} to the next larger floating-point number, and {\ttfamily dual} is the current value of the dual objective.
This is motivated by the following observation: if the step size $t_k$ leads to a decrease of $\Phi$ we get the
chain of inequalities
\[
\Phi(\xi_k) + t_k \langle d_k, \nabla \Phi(\xi_k) \rangle \le
  \Phi(\xi_k + t_k d_k)
  \le \Phi(\xi_k),
\]
where the first inequality is due to convexity. Hence, $\Phi(\xi_k + t_k d_k) $ lies in an interval of length $t_k \left|  \langle d_k, \nabla \Phi(\xi_k) \rangle \right|$.
If the condition \eqref{eq_dual_accuracy} is satisfied, then
the only floating point number in this interval is $\Phi(\xi_k)$, hence it is impossible to decrease the dual functional value any further.
In addition, \cref{cor_66} implies that $\|\nabla \Phi(\xi_k)\|_Y$ is small in this situation.

The optimal control $u = \prox_{g/\alpha}(S^*\xi/\alpha)$ computed on a mesh consisting of $20000$ triangles and with mesh-size $h = 0.014$ is depicted in \cref{fig1}.
The right-hand plot in \cref{fig1} shows $\partial\prox_{g/\alpha}(S^*\xi_k/\alpha)$, where $\xi_k$ is the final iterate of the dual variable.
In this plot, dark regions correspond to $\partial\prox_{g/\alpha}=0$, which means that the control is either zero or at the bounds $\pm R$,
while light regions  correspond to $\partial\prox_{g/\alpha}=0$, which are regions, where the control is between zero and the bounds.

\begin{figure}[htb]
\begin{center}
 \includegraphics[width=0.45\textwidth]{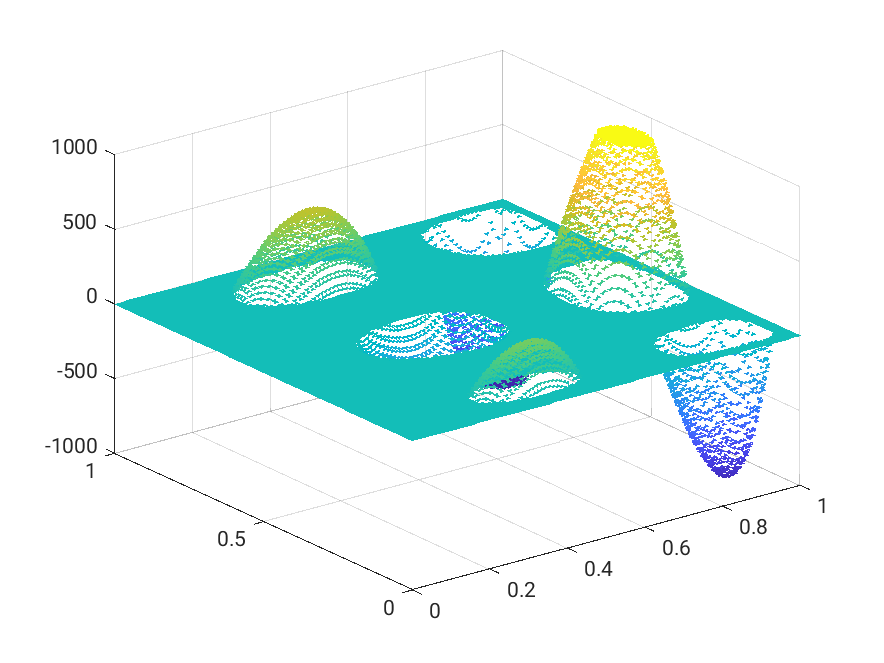}
 \includegraphics[width=0.45\textwidth]{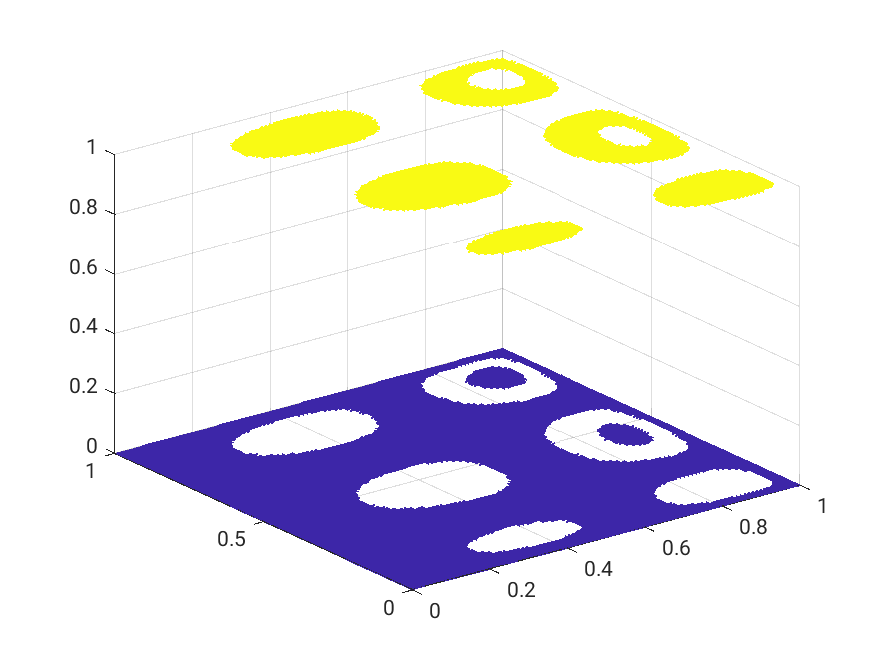}
\end{center}
\caption{Example 1: optimal control (left), $\partial\prox_{g/\alpha}$ (right)}
\label{fig1}
\end{figure}

The results for computations on different meshes can be found in \cref{table1}.
The mesh size is denoted by $h$. The finest mesh consists of $2\cdot 10^6$ triangles.
On all meshes the globalized semismooth Newton method needed 8 iterations.
The final value of the dual objective $\Phi$ together with the value of the final residual and primal-dual gap can be found in  \cref{table1} as well.
In all cases the iteration stopped because the condition \eqref{eq_dual_accuracy} was satisfied.
Since the primal-dual gap is small compared to $\Phi$, the optimal primal objective value equals $-\Phi$.
In addition, the number of iterations of the globalized semismooth Newton method (it) as well
as the total number of iterations of the inner conjugate gradient method (cg) almost stay constant for the different discretizations.

\begin{table}[htb]
\sisetup{table-alignment-mode = format, table-format = 2.2e2, table-number-alignment = left}
\begin{center}
\begin{tabular}{S[table-format = 1.2e2] S[table-format=2] S[table-format=2] S[table-format=-1.2,table-auto-round]S[table-format = 1.2e-2]S[table-format = 1.2e-2]}
\toprule
$h$ & {it}& {cg} &  $\Phi$ & {gap} & {residual} \\
\midrule
 4.42e-02 & 8 & 98 &  -3.06e+00 & 4.44e-16 & 1.47e-10 \\
 2.21e-02 & 8 & 98 &  -3.41e+00 & 2.66e-15 & 4.53e-09 \\
 1.13e-02 & 8 & 99 &  -3.60e+00 & 5.33e-15 & 1.28e-09 \\
 5.66e-03 & 8 & 99 &  -3.70e+00 & 4.44e-16 & 6.32e-10 \\
 2.83e-03 & 8 & 100 &  -3.76e+00 & 8.88e-16 & 1.83e-09 \\
 1.41e-03 & 8 & 99 &  -3.78e+00 & 4.44e-16 & 3.18e-09 \\
\bottomrule
\end{tabular}
\caption{Example 1, results for different meshes}
\label{table1}
\end{center}
\end{table}
%
%

We also performed experiments with varying parameter $\alpha$ on a fixed discretization with $h= 2.83 \cdot 10^{-3}$.
The results can be found in \cref{table2}. As one expects, the iteration number increases with decreasing $\alpha$.
The third column of \cref{table2} reports the value $\|\partial\prox_{g/\alpha}(S^*\xi/\alpha)\|_{L^1(\Omega)}$, which corresponds to the
size of the inactive set, i.e., the size of the set, where $\prox_{g/\alpha}(S^*\xi/\alpha)$ takes values not in $\{-R,0,+R\}$.

\begin{table}[htb]
\sisetup{table-alignment-mode = format, table-format = 2.2e2, table-number-alignment = left}
\begin{center}
\begin{tabular}{S[table-format = 1.2e-2]S[table-format=2] S[table-format=3] S[table-format = 1.2e-2]S[table-format=-1.2]S[table-format = 1.2e-2]S[table-format = 1.2e-2]}
\toprule
$\alpha$ & {it} & {cg} & $\|\partial\prox_{g/\alpha}\|_{L^1}$ & $\Phi$ & {gap} & {residual} \\
\midrule
 1.00e-04 & 5 & 41 &  4.57e-01 & -4.70e+00 & 2.66e-15 & 1.57e-10 \\
 1.00e-05 & 8 & 100 &  2.70e-01 & -3.76e+00 & 8.88e-16 & 1.84e-09 \\
 1.00e-06 & 11 & 153 &  8.75e-02 & -3.35e+00 & 7.11e-15 & 1.30e-07 \\
 1.00e-07 & 19 & 290 &  2.23e-02 & -3.29e+00 & 5.33e-14 & 3.58e-09 \\
 1.00e-08 & 25 & 377 &  4.15e-03 & -3.28e+00 & 1.71e-13 & 9.07e-09 \\
 1.00e-09 & 29 & 427 &  6.06e-04 & -3.28e+00 & 2.12e-12 & 1.71e-08 \\
 1.00e-10 & 72 & 723 &  1.06e-04 & -3.28e+00 & 2.09e-12 & 4.24e-09 \\
\bottomrule
\end{tabular}
\caption{Example 1, results for different $\alpha$}
\label{table2}
\end{center}
\end{table}

In addition, we conducted experiments using the variational discretization of \cite{Hinze2005} as described in \cref{sec_vd}.
We studied again varying $\alpha$ on a fixed mesh with $h= 2.83 \cdot 10^{-3}$.
The results can be found in \cref{table3}.
While the results are comparable to those in \cref{table2} it is surprising that the iteration numbers grow only moderately with decreasing $\alpha$.
So it seems that the variational discretization should be favored to the piecewise constant discretization at least for small values of $\alpha$.
This might be due to the fact that the variational discretization can better resolve the boundary of active sets.

\begin{table}[htb]
\sisetup{table-alignment-mode = format, table-format = 2.2e2, table-number-alignment = left}
\begin{center}
\begin{tabular}{S[table-format = 1.2e-2]S[table-format=2]S[table-format=3]S[table-format = 1.2e-2]S[table-format=-1.2]S[table-format = 1.2e-2]}
\toprule
$\alpha$ & {it} & {cg} & $\|\partial\prox_{g/\alpha}\|_{L^1}$ & $\Phi$ &  {residual} \\
\midrule
 1.00e-04 & 5 & 41 &  4.57e-01 & -4.70e+00 & 1.61e-10 \\
 1.00e-05 & 8 & 99 &  2.70e-01 & -3.76e+00 & 3.74e-09 \\
 1.00e-06 & 11 & 153 &  8.74e-02 & -3.35e+00 & 1.40e-07 \\
 1.00e-07 & 15 & 239 &  2.22e-02 & -3.29e+00 & 1.27e-09 \\
 1.00e-08 & 20 & 302 &  4.23e-03 & -3.28e+00 & 2.11e-09 \\
 1.00e-09 & 21 & 328 &  6.02e-04 & -3.28e+00 & 1.82e-10 \\
 1.00e-10 & 20 & 294 &  6.01e-05 & -3.28e+00 & 1.47e-06 \\
\bottomrule
\end{tabular}
\caption{Example 1, results for different $\alpha$, variational discretization}
\label{table3}
\end{center}
\end{table}

\subsection{Example 2: singular control problem}

Here, we work with box constraints, i.e.,
\[
 \tilde g(u) =   \begin{cases}
                     +\infty & \text{ if } |u| > R, \\
                     0 & \text{ if } |u| \le R \\
                    \end{cases}
\]
with $R=1$.
The desired state was set to $z := Sf$ with
\[
 f(x_1,x_2) = \chi_{ [0,0.2] }(x_1) \cdot 5 \cdot \sin(\pi x_2).
\]
Hence, the control is used to counter the unwanted perturbation $f$ in the equation.
This example is adapted from an example in \cite[Section 4.4]{Rotin}.
It is challenging to solve for small values of $\alpha$, as for $\alpha \searrow0$ the corresponding solutions are not of bang-bang type.

\begin{figure}[htb]
\begin{center}
\includegraphics[width=0.45\textwidth]{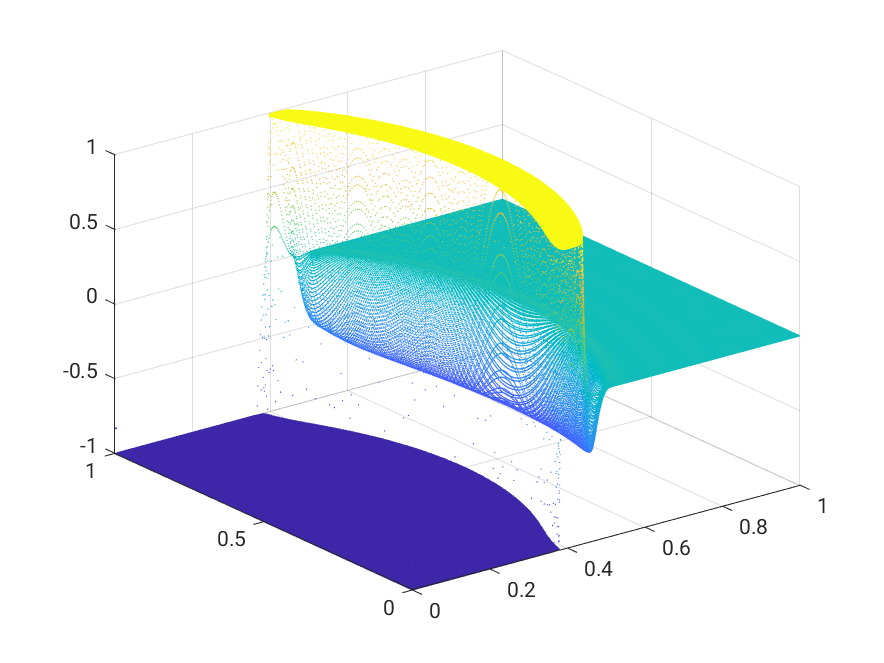}
\includegraphics[width=0.45\textwidth]{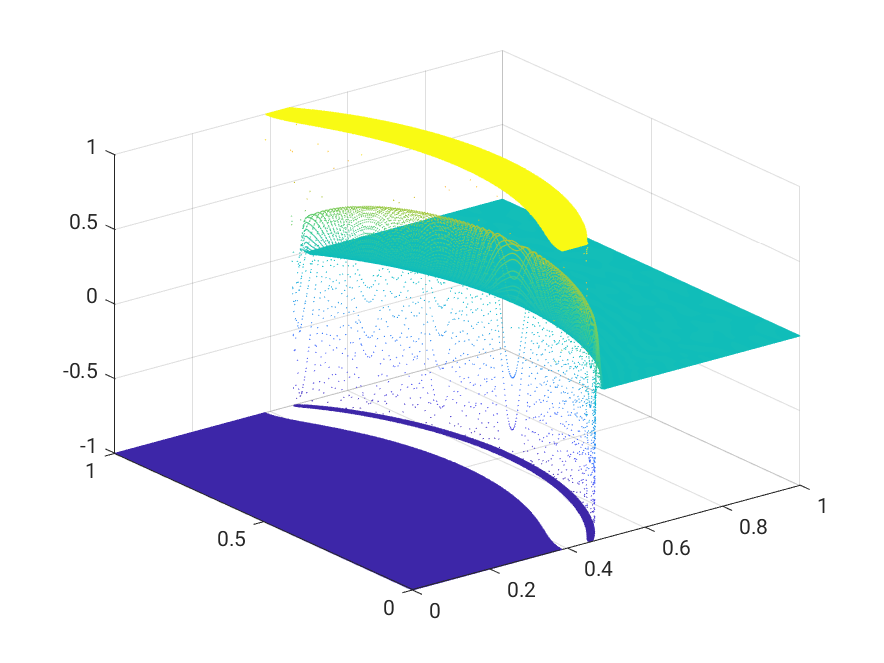}
\end{center}
\caption{Example 2: optimal controls for $\alpha =10^{-8}$  (left), $\alpha =10^{-10}$ (right)}
\label{fig2}
\end{figure}

We used the same discretization and setup as in the previous experiment.
The solution
for two different values of $\alpha \in \{ 10^{-8}, 10^{-10}\}$ and for a discretization with $h=2.83 \cdot 10^{-3}$
can be found in \cref{fig2}.
The control is zero on a large part of the domain. On the interface of the regions between $u=0$ and $u\ne0$
the control exhibits scattering behavior: it oscillates between upper and lower bounds, where
the oscillations increase with decreasing mesh size and parameter $\alpha$.
The results of computations for different values of $\alpha$
are reported in \cref{table4}.
As one can see, the iteration number grows much stronger with decreasing $\alpha$ than in the previous experiment.
Most of these iterations are spent in damping oscillations of the control variable in the region, where the control eventually is zero.
In addition, the inner CG iteration also took considerably more steps than for the previous experiment.

\begin{table}[htb]
\sisetup{table-alignment-mode = format, table-format = 2.2e2, table-number-alignment = left}
\begin{center}
\begin{tabular}{S[table-format = 1.2e-2]S[table-format=2]S[table-format=5] S[table-format = 1.2e-2]S[table-format=-1.2e-2]S[table-format = 1.2e-2]} 
\toprule
$\alpha$ & {it} & {cg} & $\|\partial\prox_{g/\alpha}\|_{L^1}$ & $\Phi$  & {residual} \\
\midrule
 1.00e-04 & 9 & 45 &  7.75e-01 & -4.61e-05 & 6.99e-18 \\
 1.00e-05 & 17 & 122 &  6.16e-01 & -3.10e-05 & 1.97e-17 \\
 1.00e-06 & 16 & 265 &  5.48e-01 & -2.90e-05 & 1.48e-11 \\
 1.00e-07 & 24 & 650 &  4.65e-01 & -2.87e-05 & 2.90e-14 \\
 1.00e-08 & 61 & 4506 &  4.43e-01 & -2.87e-05 & 7.48e-11 \\
 1.00e-09 & 52 & 8768 &  4.22e-01 & -2.87e-05 & 4.79e-14 \\
 1.00e-10 & 101 & 59713 &  4.13e-01 & -2.87e-05 & 2.63e-06 \\
\bottomrule
\end{tabular}
\caption{Example 2, results for different $\alpha$}
\label{table4}
\end{center}
\end{table}

That lead us to consider a nested iteration approach: here the iteration for some value $\alpha = 10^{-(k+1)}$ is started with the result for
the value $\alpha = 10^{-k}$. The results can be found in \cref{table5}.
There the column 'iterations' denotes the number of iterations when started with the solution of the previous iterations. For this
warm-started semismooth Newton the number of iterations per value of $\alpha$ stays constant.
In addition, the cumulated iteration numbers are still less than those for the experiment in \cref{table4},
and the inner CG iteration numbers do not grow as severely as in that setting.

\begin{table}[htb]
\sisetup{table-alignment-mode = format, table-format = 2.2e2, table-number-alignment = left}
\begin{center}
\begin{tabular}{S[table-format = 1.2e-2]S[table-format=2]S[table-format=5] S[table-format = 1.2e-2]S[table-format=-1.2e-2]S[table-format = 1.2e-2]}
\toprule
$\alpha$ & {it} & {cg} & $\|\partial\prox_{g/\alpha}\|_{L^1}$ & $\Phi$ & {residual} \\
\midrule
 1.00e-04 & 9& 45 &  7.75e-01 & -4.61e-05 & 6.99e-18 \\
 1.00e-05 & 7& 118 &  6.16e-01 & -3.10e-05 & 7.64e-18 \\
 1.00e-06 & 6& 294 &  5.48e-01 & -2.90e-05 & 7.67e-12 \\
 1.00e-07 & 7& 631 &  4.65e-01 & -2.87e-05 & 4.02e-15 \\
 1.00e-08 & 7& 1738 &  4.43e-01 & -2.87e-05 & 3.86e-09 \\
 1.00e-09 & 7& 3857 &  4.22e-01 & -2.87e-05 & 9.13e-14 \\
 1.00e-10 & 8& 11190 &  4.13e-01 & -2.87e-05 & 7.06e-06 \\
\bottomrule
\end{tabular}
\caption{Example 2, results for different $\alpha$, nested iterations}
\label{table5}
\end{center}
\end{table}

For this example, the usage of the variational discretization did not give any advantage like in the previous experiment,
so we chose not to report the corresponding numbers.

\bibliographystyle{jnsao}
\bibliography{ssn}

\end{document}